\documentclass[11pt, a4paper]{article}

\usepackage{upref,amsmath,amssymb,amsthm}
\usepackage{amsmath}
\usepackage{graphicx}
\usepackage[swedish,english]{babel}

\usepackage{amsmath,amssymb,upref,amsthm,eucal}
\usepackage{graphics}
\usepackage[swedish,english]{babel}
\newtheorem{theorem}{Theorem}
\newtheorem{lemma}[theorem]{Lemma}
\newtheorem{remark}[theorem]{Remark}
\newtheorem{proposition}[theorem]{Proposition}
\theoremstyle{remark}
\newtheorem{example}{Example}

\DeclareMathOperator{\ran}{Ran}
\DeclareMathOperator{\Ker}{Ker}

\begin{document}

\title{On the functional Hodrick-Prescott filter with non-compact operators}

\author{ Boualem Djehiche\thanks{Department of Mathematics, The Royal Institute of
Technology, S-100 44 Stockholm, Sweden. e-mail: boualem@math.kth.se}\, \thanks{Financial support from the Swedish Export Corporation (SEK) is gratefully acknowledge
}\,\, Astrid Hilbert \thanks{School of Computer Science, Physics and Mathematics, Linnaeus University, Vejdesplats 7, SE-351 95 V\"axj\"o, Sweden. e-mail: astrid.hilbert@lnu.se}\,\,  and \,  Hiba Nassar\thanks{School of Computer Science, Physics and Mathematics, Linnaeus University, Vejdesplats 7, SE-351 95 V\"axj\"o, Sweden. e-mail: hiba.nassar@lnu.se}}

\date{\today}
 \maketitle
 
\begin{abstract}
 We study a version of the functional Hodrick-Prescott filter where the associated operator is not necessarily compact, but merely closed and densely defined with closed range. We show that the associated optimal smoothing operator preserves the structure obtained in the compact case, when the underlying distribution of the data is Gaussian.

\end{abstract}
\medskip



\section {Introduction}

The study of functional data analysis is motivated by their applications in various fields of statistical estimations and statistical inverse problems (see Ramsay and Silverman (1997), Bosq (2000),  M\"uller and Stadtm\"uller (2005) and references therein). One of the most common assumptions in the studies in statistical inverse problems is to deal with compact operators. This is due to their tractable spectral properties.  The functional Hodrick-Prescott filter is often formulated as a statistical inverse problem  that reconstructs an 'optimal smooth signal' $y$ that solves an equation $Ay=v$, corrupted by a noise $v$ which is apriori unobservable, from observations $x$ corrupted by a noise $u$ which is also apriori unobservable: 
\begin{equation}\label{hod}
\left\{\begin {split}
&x= y+u, \\
&Ay=v,
\end{split}
\right.
\end{equation}
where $A: H_1\longrightarrow H_2$ is a compact operator between two appropriate Hilbert spaces $H_1$ and $H_2$.

\medskip\noindent  By introducing a smoothing operator $B$, the 'optimal smooth signal' $y(B,x)$,  associated  with $x$, is defined by 
\begin{equation}\label{HP-H-trend}
y(B,x) := \arg \min_y \left\{ \left\| x-y \right\|_{H_1}^2 +  \langle Ay, BAy\rangle_{H_2} \right\},
\end{equation} 
provided that 
$$
\langle Ah, BAh\rangle_{H_2} \ge 0,\quad h\in H_1.
$$
In \cite{djehiche3}, the optimal smoothing operator is characterized as the minimizer of the difference between the optimal smoothing signal and the best predictor of the signal given the data $x$, $E[y|x]$, when the noise $u$ and the  signal $v$ are independent Hilbert space-valued Gaussian random variables with zero means and covariance operators $\Sigma_u$ and $\Sigma_v$.

\medskip\noindent  In this paper, we extend the functional Hodrick-Prescott filter to the case where the operator $A$ is not necessarily compact. Moreover, we show that the optimal smoothing parameter preserves the structure obtained in \cite{djehiche3}, for the compact case. 

\medskip\noindent An important class of non-compact operators to which we wish to extend the  Hodrick-Prescott filter  includes the Laplace operator $A=-\frac{d^2}{dt^2}$, with Dirichlet boundary conditions, whose domain 
$$
D(A) = \{y\in H^2 ([0,1]),\;\; y(0) = y(1) =0\},
$$
where,  $H^2([0,1])$ is the Sobolev space of functions whose weak derivatives of order less than or equal to two belong to $ L^2([0,1])$ (see e.g. \cite{Dautray} for further examples).

\medskip\noindent
The paper is organized as follows. In Section 2, we generalize the functional Hodrick-Prescott filter under the assumption that the operator $A$ is closed and densely defined with closed range. In Section 3, we prove that the optimal smoothing operator maintains same form when the covariance operators $\Sigma_u$  and $\Sigma_v $ are trace class operators. In Section 4, we illustrate this filter with two examples.  In Section 5, we extend this characterization to the case where the covariance operators $\Sigma_u$  and $\Sigma_v $ are not trace class such e.g. white noise.

\section{A Functional Hodrick-Prescott filter with closed operator} 

Let $H_1$ and $H_2$ be two separable Hilbert spaces, with norms $\left\| \cdot\, \right\|_{H_i}$ and inner products $\langle \cdot ,\cdot \rangle _{H_i},\,\, i=1,2$, and $x \in H_1$ be a functional time series of observables.  Given a linear operator  $A:H_1 \rightarrow H_2$, the Hodrick-Prescott filter extracts an 'optimal smooth signal' $y\in  H_1$ that solves an equation $Ay=v$, corrupted by a noise $v$ which is apriori unobservable, from observations $x$ corrupted by a noise $u$ which is also apriori unobservable: 
\begin{equation}\label{hod}
\left\{\begin {split}
&x= y+u, \\
&Ay=v.
\end{split}
\right.
\end{equation}
\noindent Optimality of the extracted signal is achieved by the following Tikhonov-Phillips regularization of the system (\ref{hod}), by introducing a linear operator $B :H_2 \rightarrow H_2$ which acts as a smoothing parameter:
\begin{equation}\label{HP-H-trend}
y(B,x) := \arg \min_y \left\{ \left\| x-y \right\|_{H_1}^2 +  \langle Ay, BAy\rangle_{H_2} \right\},
\end{equation}
provided that 
\begin{equation}\label{positive}
\langle Ah, BAh\rangle_{H_2} \ge 0,\quad h\in H_1.
\end{equation}

\medskip\noindent As suggested in \cite{djehiche3}, the optimal smoothing operator minimizes the gap between the conditional expected value of $y$ given $x$, $E[y|x]$, which is the best predictor of $y$ given $x$, and $y(B,x)$:
\begin{equation}\label{B-best}
\hat B = \arg \min_B \left\| E[y|x] - y(B ,x) \right\|_{H_1}^2.
\end{equation}

\medskip\noindent The main purpose of this work is to extend the characterization of the optimal smoothing operator obtained in  \cite{djehiche3} to the case where the linear operator $A$ is not necessarily compact, and $u$ and $v$ are independent Hilbert space-valued generalized Gaussian random variables with zero means and covariance operators $\Sigma_u$ and $\Sigma_v$. 

\medskip\noindent
We  assume that the linear operator $A:H_1 \rightarrow H_2$ is

\medskip
{\bf Assumption (1)} 
\begin{enumerate}
\item[(1a)] Closed and defined on a dense subspace $\mathcal{D} (A)$ of $H_1$, 

\item[(1b)] Its  range, $\ran (A)$, is closed.
\end{enumerate}

\noindent Its Moore-Penrose generalized inverse  $A^\dagger$ is defined on  
$$ 
\mathcal{D}(A^\dagger ) = \ran (A)+ \ran (A)^\bot,  
$$ 
i.e. $\Ker(A ^\dagger) = \ran( A)^\bot$. 

\medskip\noindent Assumption (1) is equivalent to the fact that $A ^\dagger$ is bounded (see \cite{Groetsch}, \cite{kato}). 

\medskip\noindent For  all $v \in \mathcal{D} (A^\dagger ) $, the set of all solutions of the equation 
$$
Ay=v,\qquad y\in \mathcal{D}(A),
$$ 
is given by 
$$ \{ y ^\dagger + y_0; \quad y_0 \in \Ker(A) \},$$
where $y ^\dagger$ is the unique minimal-norm solution given by $y^\dagger = A^\dagger v.$

\medskip\noindent Hence, for arbitrary  $y_0 \in \Ker(A)$, we have
\begin{equation}\label{solution1}
 y= y_0 + A^\dagger v,
\end{equation}  
and,  in view of (\ref{hod}),
\begin{equation}\label{solution2}
 x= y_0 + A^\dagger v + u.
\end{equation}

\medskip\noindent Let $\Pi:= A^\dagger A$. By the Moore-Penrose equations, we have $\Pi A^\dagger= A^\dagger$, $\Pi^2=\Pi$ and $\Pi^*=\Pi$ (self-adjoint). Therefore, $\Pi$ an orthogonal projector. It is easily checked that, 
for every $\xi\in H_1$,  the elements $\Pi\xi$ and $(I_{H_1}-\Pi)\xi$ are orthogonal:
\begin{equation}\label{Pi-2}
<\Pi\xi,(I_{H_1}-\Pi)\xi >=0,
\end{equation}
and
\begin{equation}\label{Pi}
\mbox{Ker}(A)=\mbox{Ker}(\Pi)=\mbox{Ran}(I_{H_1}-\Pi).
\end{equation}
Moreover, we have $(I_{H_1}-\Pi)y=y_0$ and $A^\dagger v=\Pi y$.

\medskip\noindent In the next proposition we show that the problem (\ref{HP-H-trend}) has a unique solution for a class of linear smoothing operators $B$ satisfying (\ref{positive}).

\begin{proposition}\label{mainprop}
Let $ A:H_1 \longrightarrow H_2$ be a closed, linear operator and its domain is dense in $H_1$.  Assume further the smoothing operator $B: H_2 \longrightarrow H_2$ is closed, densely defined and satisfies
\begin{equation}\label{B}
\langle Ah, BAh\rangle_{H_2} \ge 0,\quad h\in H_1.
\end{equation}
Then, there exists a unique $y(B ,x) \in H_1$ which minimizes the functional 
\[  J_B (y) = \left\| x-y \right\|^2_{H_1} + \langle Ay, BAy\rangle_{H_2}.\]
This minimizer is given by the formula 
\begin{equation}\label{mini}
 y(B, x) = (I_{H_1}+ A^*BA)^{-1} x. 
\end{equation}
\end{proposition}

\begin{proof}
It is immediate to check that the minimizer of the functional 
$\left\| x-y \right\|_{H_1}^2 +  \langle Ay, BAy\rangle_{H_2}$
is $( I_H+ A^*BA)^{-1} x$ provided that the function $( I_H+ A^*BA)^{-1}$ exists everywhere. But, since the operator $D:= \sqrt{B} A$ is closed and densely defined, thanks to a result by Neumann (see \cite{riesz}, Sec. 118, Chap. VII),  $( I_H+ A^*BA)^{-1}=(I+D^*D)^{-1}$ exists everywhere and is bounded. This finishes the proof of the proposition.
\end{proof}

\section{HP filter associated with trace class covariance operators}

In this section we prove that the optimal smoothing operator which solves (\ref{B-best}) has the same structure as in \cite{djehiche3}, when $u$ and $v$ are independent Gaussian random variables with zero mean and trace class covariance operators $\Sigma_u$ and $\Sigma_v$.

\medskip\noindent In view of (\ref{solution1}) and (\ref{solution2}), a stochastic model for $(x,y)$ being determined by models for $y_0$ and $(u,v)$, we assume 

\medskip{\bf Assumption (2)} $y_0$ deterministic.

\medskip{\bf Assumption (3)}  $u$ and $v$ are independent Gaussian random variables with zero mean and covariance operators $\Sigma_u$ and $\Sigma_v$ respectively.   

\medskip\noindent Assumption (2) is made to ease the analysis. The independence between $u$ and $v$ imposed in Assumption (3) is natural because apriori there should not be any dependence between the 'residual' $u$ which is due to the noisy observation $x$ and the required degree of smoothness of the signal $y$.

\medskip\noindent Assumption (3) implies that $\Pi y=A^\dagger v$ and $u$ are also independent. Thus, with regard to the following decomposition of $x$,
$$
x=y_0+\Pi y+\Pi u+(I_{H_1}-\Pi)u,
$$
it is natural to assume that even the orthogonal random variables $\Pi u$ and $(I_{H_1}-\Pi)u$ independent. This would mean that the input $x$ is decomposed into three independent random variables. This is actually the case for the classical HP filter. Also, as we will show below, thanks to this property the optimal smoothing operator has the form of a 'noise to signal ratio' in line with the classical HP filter.  

\medskip {\bf Assumption (4)} The orthogonal (in $ H_1$) random variables $\Pi u$ and $(I_{H_1}-\Pi)u$ are independent:
\begin{equation}\label{indep-1}
\Pi\Sigma_u=\Sigma_u\Pi.
\end{equation}

\medskip\noindent We note that (\ref{indep-1}) is equivalent to 
\begin{equation}\label{indep-2}
\Pi\Sigma_u\Pi=\Pi\Sigma_u.
\end{equation}

\medskip\noindent Given Assumptions (2) and (3), by (\ref{solution1}) and (\ref{solution2}), it holds that $(x,y)$ is Gaussian with  
 mean $(E[x],E[y])=(y_0, y_0)$, and covariance operator
\begin{equation}\label{cov(x,y)}
\begin{split}
\Sigma &= \left( \begin{array}{ccc}
\Sigma _u + Q_v  && Q_v \\ 
Q_v &&  Q_v 
\end{array} 
\right), 
\end{split}
\end{equation}
where,
\begin{equation}\label{Qv}
Q_v:= A^\dagger\Sigma_v(A^\dagger)^*.
\end{equation}

\begin{lemma}
\label{THS}
The linear operator $Q_v$ is trace class. 

\medskip \noindent Moreover, the linear operator
$$
T:=Q_v\left[\Sigma_u+Q_v\right]^{-1/2}
$$
is Hilbert-Schmidt.
\end{lemma}

\begin{proof}
Since the covariance operator $\Sigma_v$ is a trace class operator, $\Sigma_v^{\frac{1}{2}}$ is Hilbert-Schmidt. Therefore,  $A^\dagger \Sigma_v^{\frac{1}{2}}$ is Hilbert-Schmidt, since, by Assumption (1), $A^\dagger$ is bounded.  Hence,
$$
Q_v=A^\dagger \Sigma_v^{\frac{1}{2}} (A^\dagger \Sigma_v^{\frac{1}{2}})^*
$$
as a product of two Hilbert-Schmidt operators, is a trace class operator. Furthermore, since $\Sigma_u+Q_v$ is injective and trace-class, the operator $\left[\Sigma_u+Q_v\right]^{-1/2}$ is Hilbert-Schmidt. Hence, $T:=Q_v\left[\Sigma_u+Q_v\right]^{-1/2}$ is Hilbert-Schmidt.
\end{proof} 

\medskip\noindent We may apply Theorem 2 in \cite{Mandelbaum} to obtain the conditional expectation of the signal $y$ given the functional data $x$:
\begin{equation}\label{cond-H}
E[y|x] =y_0+Q_v\left[\Sigma_u+Q_v\right]^{-1}(x-y_0).
\end{equation}

\medskip\noindent The following theorem is a generalization of Theorem 4 in \cite{djehiche3}. 
\begin{theorem}\label{optimalb}
Under Assumptions (1), (2) (3) and (4), the smoothing operator 
\begin{equation}\label{B-hat}
\hat B:=(A^\dagger )^*\Sigma_u A^*\Sigma_v^{-1}
\end{equation}
is the unique operator which satisfies
\[ \hat B = \arg\min_B\left\|E[y|x]- y(B ,x) \right\|_{H_1}, \]
where the minimum is taken with respect to all linear closed and densely defined operators which satisfy the positivity condition (\ref{B}).
\end{theorem}

\begin{proof} The proof is similar to that of Theorem 4 in \cite{djehiche3}.
\end{proof}

\section{Examples}

In this section, we apply Theorem 3 to two examples for which the operators are densely defined with closed range.

\begin{example} Inspired by an example discussed in \cite{Kulkarni1}, we consider the operator $A:l^2 \rightarrow l^2$ defined by
\begin{equation}
\label{Aexample}
A(x_1, x_2, x_3,\ldots, x_n,\ldots) = (0, 2x_2, 3x_3,\ldots, nx_n,\ldots),
\end{equation}
with domain  
$$
\mathcal {D} (A) = \{x:=(x_1, x_2, x_3,\ldots, x_n,\ldots) \in l^2 : \sum _{j=1} ^\infty {|jx_j|^2 } < \infty \}.
$$
This operator is self-adjoint, unbounded, closed and densely defined with $ \overline{ \mathcal{D} (A)} =l^2 $.

\medskip\noindent Now consider the Hodrick-Prescott filter associated the operator $A$. 
Under the assumption that $u$ and $v$ are independent Gaussian random variables with zero means and covariance operators of trace class of the form
$$
\Sigma_u x = (\sigma_1^u x_1, \sigma_2^u x_2, \sigma_3^u x_3,\ldots, \sigma_n^u x_n, \ldots),
$$
and
$$
\Sigma_v x = (\sigma_1^v x_1, \sigma_2^v x_2, \sigma_3^v x_3,\ldots, \sigma_n^v x_n, \ldots),
$$
respectively.

\medskip\noindent In view of the form of the operator $A$, an appropriate class of smoothing operator $B$ is
$$
B(x_1, x_2, x_3,\ldots, x_n,\ldots) = (b_1 x_1, b_2x_2, b_3x_3,\ldots,b_n x_n,\ldots),
$$
where, the coefficients $\{b_n,\,\, n=1,2,\ldots\}$ are chosen so that the operator $B$ is closed, densely defined and satisfies the positivity condition (\ref{B}).
In view of Theorem \ref{optimalb}, the optimal smooth operator $B$ given by (\ref{B-hat})  reads
$$
\hat{B} x = A^{-1} \Sigma_u A \Sigma_v^{-1} x = (0, \frac{\sigma_2 ^u}{\sigma_2 ^v} x_2, \frac{\sigma_3 ^u}{\sigma_3 ^v} x_3,\ldots, \frac{\sigma_n ^u}{\sigma_n ^v} x_n,\ldots),\quad x\in l^2.
$$
Moreover, the corresponding optimal signal given by (\ref{mini}) is
$$
 y(\hat B, x) = (I_{H_1}+ A^*{\hat B}A)^{-1} x = (x_1, \frac{1}{4\hat b_2 +1}  x_2, \frac{1}{9\hat b_3 +1} x_3,\ldots, \frac{1}{n^2 \hat b_n +1} x_n,\ldots),
$$
where, $\hat b_j:=\sigma_j ^u/\sigma_j ^v,\,\,\, j=1,2,\ldots$.
\end{example}

\begin{example} Consider the Laplace operator $A=-\frac{d^2}{dt^2}$, with Dirichlet boundary conditions,  whose domain is 
$$
D(A) = \{y\in H^2 ([0,1]),\;\; y(0) = y(1) =0\},
$$
where,  $H^2([0,1])$ is the Sobolev space of functions whose weak derivatives of order less than or equal to two belong to $ L^2([0,1])$. 

\medskip\noindent 
The signal process $y$ corrupted by $v$ satisfies 
\begin{equation} \label{exhod}
Ay(t) = - \frac{d^2 y(t)}{dt^2}=v(t), \quad y \in D(A),
\end{equation}

\medskip\noindent The Laplacian  $A$ is one-to-one, non-negative, self-adjoint, closed, unbounded operator, with domain $D(A)$ dense in $L^2([0,1])$ (see e.g. \cite{Dautray}). The eigenvalues and eigenvectors of $A$ satisfy: 
\[
	\begin{cases}
	\begin{split}
	 &\lambda_ n = n^2 \pi ^2, \qquad n \geq 1 \ \ \  \mbox{ and } \ \ \ n \in \mathbb{N}\\
	 & e_n (t) = \sqrt {2} \sin n \pi t.
	\end {split} 
	\end{cases} 
	\]

	The inverse of the operator $A$ is a self-adjoint Hilbert-Schmidt operator and given by 
	$$
	(A^{-1} x) (t) = \int _0 ^1 {G(t,s) x(s)ds }
	$$
	where the Green function $G: [0,1] \times [0,1] \rightarrow [0,1]$ is given by 
	\[
	G(t,s):= 
	\begin{cases}
	\begin{split}
	 &(1-t) s, \qquad 0 \leq s \leq t, \\
	 & t(1-s), \qquad t \leq s \leq 1.
	\end {split} 
	\end{cases} 
 \]
Hence, the operator $A$ can be written 
$$
Ay(t) =  \sum _{n=1}^\infty {n^2 \pi ^2  \langle y, e_n \rangle e_n(t)}, 
$$
and the solution of equation (\ref{exhod}) is given in terms of eigenvalues and eigenvectors by
$$
y(t) = A^{-1}v(t) =  \sum _{n=1}^\infty { \frac{1}{ n^2 \pi ^2 } \langle v ,  e_n \rangle  e_n(t)}. 
$$

\medskip\noindent Now consider the Hodrick-Prescott filter associated the operator $A$. 
Under the assumption that $u$ and $v$ are independent Gaussian random variables with zero means and covariance operators of trace class of the form
$$
\Sigma_u h(t) = \sum _{n=1}^\infty {\sigma_n ^u  \langle h ,  e_n \rangle  e_n(t)}, 
$$
and
$$
\Sigma_v h(t) = \sum _{n=1}^\infty {\sigma_n ^v  \langle h , e_n \rangle  e_n(t)},
$$
respectively, where the sums converge in the operator norm. 
\medskip\noindent The smoothing operator $B$  is defined as 
$$
B h(t) = \sum _{n=1}^\infty {\beta _n \langle h ,  e_n \rangle  e_n(t)},
$$
where, the coefficients $\{\beta_n,\,\, n=1,2,\ldots\}$ are chosen so that the operator $B$ is closed, densely defined and satisfies the positivity condition (\ref{B}). By Theorem \ref{optimalb}, the optimal smooth operator $B$ given by (\ref{B-hat})  reads  
$$
\hat{B} h(t) = A^{-1} \Sigma_u A \Sigma_v^{-1} h(t) = \sum _{n=1}^\infty {\frac{  \sigma_n ^u }{\sigma _n ^v}  \langle h , e_n \rangle  e_n(t)}.
$$
The corresponding optimal signal given by (\ref{mini}) is
$$
 y(\hat B, x) = (I_{L^2(0,1)}+ A^*\hat BA)^{-1} x =  \sum _{n=1}^\infty { \left( 1+ n^4 \pi ^4 \frac{  \sigma_n ^u }{\sigma _n ^v} \right)^{-1} \langle x, e_n\rangle  e_n(t)}.
$$ 

\end{example}

\section{Extension to non-trace class covariance operators} 

In this section we show that the characterization (\ref{B-hat}) of the optimal smoothing operator is preserved even when the covariance operators of $u$ and $v$ are not necessarily trace class operators.

\medskip\noindent Assuming that $u \sim N(0,\Sigma _u)$ and $v \sim N(0,\Sigma _v)$ where $\Sigma _u$ and $ \Sigma _v$ are self-adjoint positive-definite bounded but not trace class operators on $H_1$ and $H_2$, respectively. One important case of this extension is where $u$ and $v$ are white noise with covariance operators of the form $\Sigma_u=\sigma_u^2I_{H_1}$ and $\Sigma_v=\sigma_v^2I_{H_2}$, respectively, for some constants $\sigma_u$ and $\sigma_v$.

\medskip\noindent Following Rozanov (1968) (see also Lehtinen {\it et al.} (1989)),  we consider these Gaussian variables as generalized random variables on an appropriate Hilbert scale (or nuclear countable Hilbert space), where the covariance operators can be maximally extended to self-adjoint positive-definite, bounded and trace class operators on an appropriate domain. 

\medskip\noindent  We first construct the Hilbert scale  appropriate to our setting. This is performed using the linear operator $A$ as follows (see Engle {\it et al.} (1996) for further details).

\medskip\noindent In view of Assumption (1), the operator $A^\dagger: H_2 \rightarrow H_1$ is  linear and bounded operator. Put $H_3 := \ran{A}$,  $H_3$ is a Hilbert space, since it is a closed subspace of Hilbert space $H_2$. Let $\bar{A} ^\dagger $ be the restriction of $A^\dagger$ on $H_3$ i.e.  $\bar{A}^\dagger: H_3 \rightarrow H_1$. Hence $\bar{A} ^\dagger $ is injective bounded linear operator. 

\begin{remark}
In view of Hodrick-Prescott Filter (\ref{hod}), $v \in  \ran (A) = H_3$ i.e. it can be seen as $H_3$-random variable with covariance operator $\Sigma_v :H_3 \rightarrow H_3$.
\end{remark}
Set 
$$
K_1:= (\bar{A}^\dagger (\bar{A}^\dagger)^* )^{-1}: H_1 \rightarrow H_1. 
$$

\medskip\noindent We can define the fractional power of the operator $K_1$ by 
\begin{equation}
\label{s-K1}
K^s_1 h = (\bar{A}^\dagger (\bar{A}^\dagger)^* )^{-s} h,\quad h\in H_1,\quad s\ge 0,
\end{equation}
and define its domain by
\begin{equation}
\mathcal{D}(K_1^{s}):=\left\{h\in H_1; \quad (\bar{A}^\dagger (\bar{A}^\dagger)^* )^{-s} h \in H_1 \right\}.
\end{equation}

\medskip\noindent Let  $ \mathcal{M} $ be the set of all elements $x$ for which all the powers of $K_1$ are defined i.e. 
$$ \mathcal{M} := \bigcap \limits _{n=0} ^\infty \mathcal {D} (K_1^{n}).$$
For $s \geq 0$, let $H_1^s$ be the completion of $ \mathcal{M}$  with respect to the Hilbert space norm  induced by the inner product 
\begin{equation}
 \langle x,y\rangle _{H_1^s} :=  \langle K_1^s x,K_1^s y\rangle_ {H_1}, \quad x,y \in \mathcal{M},
\end{equation}
and let  $H_1 ^{-s} := (H_1^s)^*$ denote the dual of $H_1^s$ equipped with the following inner product: 
\begin{equation}
 \langle x,y\rangle_{H_1 ^{-s}} :=  \langle K_1^{-s}x,K_1^{-s} y\rangle_{H_1}, \quad x,y \in \mathcal{M}.
\end{equation}
Then,  $(H_1^s)_{s \in \mathbb{R}}$  is the Hilbert scale induced by the operator $K_1$. 

\medskip\noindent The operator $K_2:= ((\bar{A}^\dagger)^* \bar{A}^\dagger)^{-1}: H_3 \rightarrow H_3$  has the same properties as $K_1$.
Repeating the same procedure as before we get that $(H_3^s)_{s\in \mathbb{R}}$  is the Hilbert scale induced by the operator $K_2$, where the norm in $H_3^n$ is given by  $\left\| h \right\| _{H_3^{n}} = \left\| K_2 ^{n} h \right\|_{H_3} $,  $\quad h \in H_3^n.$\\

\medskip\noindent Noting that
\[ H_1^{-n} = \mbox{Im} \left( (\bar{A}^\dagger (\bar{A}^\dagger)^* )^n \right) = (\bar{A}^\dagger (\bar{A}^\dagger)^* )^n (H_1), \]
  
\[ H_3^{-n} = \mbox{Im} \left( ((\bar{A}^\dagger)^* \bar{A}^\dagger)^n \right) = ((\bar{A}^\dagger)^* \bar{A}^\dagger)^n (H_3) \]
with $\ker \left( (\bar{A}^\dagger (\bar{A}^\dagger)^* )^n \right) = \ker (\bar{A} ^\dagger) = \{0\}$ and  $\ker \left( ((\bar{A}^\dagger)^* \bar{A}^\dagger)^n \right) = \ker ((\bar{A} ^\dagger ) ^*) = \ker ( \bar{A}) = \ker (A)$, it follows that the operator $\bar{A} ^\dagger$ extends to a continuous operator from $H_3^{-n}$ into $ H_1^{-n}$, and the operator $ \bar {A} = A $ extends to a continuous operator from $ H_1^{-n}$  into $H_3^{-n}$, and the operators  $\bar{A} ^\dagger (\bar{A} ^\dagger)^* $ and  $(\bar{A} ^\dagger)^* \bar{A} ^\dagger$ extend as well to a continuous operator onto $H_1^{-n}$ and $H_3^{-n}$ respectively.\\

\medskip\noindent 
Extending the HP filter to  the larger Hilbert spaces $H_1^{-n}$ and $H_3^{-n}$ due to the flexibility offered by the Hilbert scale, where $n$ is chosen so that  the second moments $ E[\|x\|^2_{H_1 ^{-n}}], E[\|y\|^2_{H_1 ^{-n}}]$, $E[\|u\|^2_{H_1 ^{-n}}]$ and $E[\|v\|^2_{H_3^{-n}}]$
of the Gaussian random variables $x, y, u$ and $v$  in $H_1^{-n}$ and $H_3^{-n}$ respectively, are finite. This amounts to make their respective covariance operators 

\begin{equation}
\label{tildeuv}
\tilde{\Sigma }_u =  (\bar{A} ^\dagger (\bar{A} ^\dagger)^* )^n \Sigma _u  ( \bar{A}^\dagger ( \bar{A}^\dagger)^* )^n, \quad \tilde{\Sigma }_v = ((\bar{A} ^\dagger)^* \bar{A} ^\dagger)^n \Sigma _v ((\bar{A} ^\dagger)^* \bar{A} ^\dagger)^n
\end{equation}
and 
\begin{equation}
\tilde\Sigma= \left( \begin{array}{ccc}
\widetilde{\Sigma} _u + \widetilde{Q}_v  && \widetilde{Q}_v \\ 
\widetilde{Q}_v &&  \widetilde{Q}_v 
\end{array} 
\right), 
\end{equation}
where,
\begin{equation}\label{Qv}
\widetilde{Q}_v:=  \bar{A}^\dagger \widetilde{\Sigma}_v (\bar{A} ^\dagger)^*.
\end{equation}
trace class.
\medskip We make the following assumption:

\medskip{\bf Assumption (5).} There is $n_0 >0 $ such that the covariance operators $\tilde\Sigma_u, \tilde\Sigma$ and $ \tilde{\Sigma}_v$  are trace class on the Hilbert spaces $H_1^{-n}$ and  $H_3^{-n}$, respectively.

\medskip\noindent It is worth noting that since $y_0 \in \ker ({A}) = \ker (((\bar{A}^\dagger)^* \bar{A}^\dagger)^n )$ then  $\left\| y_0 \right\|_ {H_1^{-n}} = \left\| ( \bar{A}^\dagger ( \bar{A}^\dagger)^* )^n y_0 \right\|_ {H_1} =0$. Hence, the $H_1^{-n}\times H_1^{-n}$-valued random vector $(x,y)$ has mean $(E[x],E[y])=(0, 0)$.

\medskip\noindent  Summing up, by Assumption 5, for $n\ge n_0$, the vector $(x,y)$ is an $H_1^{-n}\times H_1^{-n}$-valued Gaussian vector with mean $(0,0)$ and covariance operator $\tilde\Sigma$.  Thus, by Theorem 2 in \cite{Mandelbaum}, we have 
\begin{equation} 
E[y|x] =\widetilde{Q}_v\left[\widetilde{\Sigma}_u+\widetilde{Q}_v\right]^{-1}x, \qquad \mbox{a.s. in } H_1^{-n}.
\label{expn}
\end{equation}
provided that the operator 
\begin{equation}\label{HS-N}
\tilde T := \widetilde{\Sigma}_{XY} \widetilde{\Sigma} _X^{-\frac{1}{2}}
\end{equation}
is Hilbert-Schmidt. But, in view of Assumption (5) and Lemma (\ref{THS}), the operator $T$ is Hilbert-Schmidt.

\medskip\noindent
The deterministic optimal signal associated with $x$ in $H_1^{-n},  n \geq n_0$, is given by the formula (cf. Proposition \ref{mainprop})  
\begin{equation}
y(B,x) = ( I_{H_1^{-n}} + A^* BA ) ^{-1} x,
\label{opty}
\end{equation}
which is the unique minimizer of the functional 
\begin{equation}\label{HP-HN-trend}
J_B(y) =  \left\| x-y \right\|_{H_1^{-n}}^2 +  \langle Ay, BAy\rangle_{H_3^{-n}} ,
\end{equation} 
with a linear  operator $B: H_3^{-n}\longrightarrow H_3^{-n}$ such that $\langle Ah, BAh\rangle_{H_3^{-n}} \ge0 $ for all $h \in H_1^{-n}$. 
 
\medskip\noindent The following theorem  gives an explicit expression of the optimal smoothing operator $\hat{B} $. 

\begin{theorem}\label{optB}
Let assumption 5 hold. Then, the unique optimal  smoothing operator associated with the HP filter associated with $H_1^{-n}$-valued data $x$ is given by:
\begin{equation}\label{hat-BN}
\hat Bh:= ( \bar{A}^\dagger )^* \tilde\Sigma_u A^*\tilde{\Sigma}_v^{-1}h, \quad h\in H_3^{-n}.
\end{equation}
\end{theorem}

\medskip\noindent The proof is similar to that of Theorem 6 in  \cite{djehiche3}.

\subsection{The white noise case-  Optimality of the noise-to-signal ratio}
In this section we  show that the optimal smoothing operator $\hat B$ given by (\ref{hat-BN}) reduces to  the noise-to-signal ratio  where  $u$ and $v$ are white noises. Assuming  $u$ and $v$  independent and  Gaussian random variables with zero means and covariance operators  $\Sigma_u = \sigma_u I_{H_1 }$ and $\Sigma_v = \sigma_v I_{H_3 }$, where $I_{H_1 }$ and $I_{H_3 }$ denote the $H_1$ and $H_3$ identity operators, respectively and $\sigma_u$ and $\sigma_v$ are constant scalars. Assumption 5 reduces to 

\medskip{\bf Assumption 6.} There is an $n_0>0$ such that  $\left( \bar{A}^\dagger ( \bar{A}^\dagger )^*\right)^{2n}$ and $\left(( \bar{A}^\dagger )^*  \bar{A}^\dagger \right)^{2n}$ are trace class for all $n \ge n_0$.\\

\medskip\noindent Under this assumption, the  associated covariance operators 
\begin{equation*}
\begin{split}
\tilde{\Sigma }_u &= (\bar{A} ^\dagger (\bar{A} ^\dagger)^* )^n \Sigma _u  ( \bar{A}^\dagger ( \bar{A}^\dagger)^* )^n = \sigma_u \left( \bar{A}^\dagger ( \bar{A}^\dagger )^*\right)^{2n},\\
\tilde{\Sigma }_v & = ((\bar{A} ^\dagger)^* \bar{A} ^\dagger)^n \Sigma _v ((\bar{A} ^\dagger)^* \bar{A} ^\dagger)^n = \sigma_v \left(( \bar{A}^\dagger )^*  \bar{A}^\dagger \right)^{2n}
\end{split}
\end{equation*}
and 
\begin{equation*}
\widetilde{Q}_v=  \sigma_v A^\dagger \left((A^\dagger )^* A^\dagger \right)^{2n} (A^\dagger )^* = \sigma_v  \left(A^\dagger (A^\dagger )^*\right)^{2n+1} 
\end{equation*}
are trace class, the expression (\ref{hat-BN}) giving the optimal smoothing operator $\hat B$ reduces to 

 \begin{equation}
\hat B = ( \bar{A} ^\dagger )^* \tilde\Sigma_u A^*\tilde{\Sigma}_v^{-1}h = \frac {\sigma_u}{\sigma_v} I_{H_3^{-n}},
\end{equation}
i.e. $\hat B$ is the noise-to-signal ratio which is in the same pattern as in the classical HP filter.

%\clearpage

\end{document}